\documentclass[11pt,fleqn]{article}

\usepackage{amsfonts,amsmath,amssymb,amsthm,authblk,graphicx,latexsym,mathrsfs,subfig}

\newtheorem{theorem}{Theorem}

\newcommand{\beeq}{\begin{equation}}
\newcommand{\eneq}{\end{equation}}
\newcommand{\bear}{\begin{eqnarray}}
\newcommand{\enar}{\end{eqnarray}}
\newcommand{\bearno}{\begin{eqnarray*}}
\newcommand{\enarno}{\end{eqnarray*}}

\newcommand{\vol}{{\rm vol}}

\renewcommand\footnotemark{}

\begin{document}

\title{Spectral analysis of communication networks using Dirichlet eigenvalues}

\author[1]{Alexander Tsiatas}
\author[2]{Iraj Saniee\thanks{This work was supported by AFOSR Grant No. FA 9550-08-2-0064.}}
\author[3]{Onuttom Narayan}
\author[2]{Matthew Andrews}
\affil[1]{Department of Computer Science and Engineering, University of California, San Diego, 9500 Gilman Drive, La Jolla, CA 92093-0404}
\affil[2]{Mathematics of Networks, Bell Laboratories, Alcatel-Lucent, 600 Mountain Avenue, Murray Hill, NJ 07974}
\affil[3]{Department of Physics, University of California, Santa Cruz, CA 95064}

\date{}

\maketitle

\begin{abstract}
The spectral gap of the graph Laplacian with Dirichlet boundary conditions is
computed for the graphs of several communication networks at the IP-layer,
which are subgraphs of the much larger global IP-layer network. We show that
the Dirichlet spectral gap of these networks is substantially larger than the
standard spectral gap and is likely to remain non-zero in the infinite graph
limit. We first prove this result for finite regular trees, and show that the
Dirichlet spectral gap in the infinite tree limit converges to the spectral
gap of the infinite tree. We also perform Dirichlet spectral clustering on the
IP-layer networks and show that it often yields cuts near the network core
that create genuine single-component clusters. This is much better than
traditional spectral clustering where several disjoint fragments near the
periphery are liable to be misleadingly classified as a single cluster.
Spectral clustering is often used to identify bottlenecks or congestion;
since congestion in these networks is known to peak at the core, our results
suggest that Dirichlet spectral clustering may be better at finding bona-fide
bottlenecks. 
\end{abstract}

\section{Introduction}
\label{intr}

Many real-world networks are truly vast, encompassing millions or billions
of nodes and edges, e.g., social and biological networks. This scale
produces computational challenges: the large majority of algorithms
are too computationally intensive to use at this scale on general
graphs. Instead, one can study smaller sub-graphs of these networks;
for example, the portion of a social network corresponding to one
university, or the portion of a communication network corresponding
to one Internet service provider. 

Spectral graph theory \cite{book}, the study of eigenvalues and
eigenvectors of graph-theoretic matrices, is often used to analyze
various graph properties.  One might hope that the properties of a
large sub-graph of a network will be representative of the properties
of the entire network. Unfortunately, the properties of an expander
graph depend on the conditions imposed at its (large) boundary. In
particular, the spectral gap of the graph Laplacian on a finite
truncation of an infinite regular tree approaches zero as the size
of the truncation is increased, even though the spectral gap of its
infinite counterpart is non-zero. In this paper we show that, by
contrast, if the spectral gap is calculated with {\it Dirichlet boundary 
conditions,} it approaches the infinite graph limit as the size of 
truncation is increased. 

Motivated by this result, we compute the Dirichlet spectral gap for
ten IP-layer communication networks as measured and documented by
previous researchers in the Rocketfuel database \cite{rf}. We find
that the Dirichlet spectral gap is much larger than the traditional
spectral gap for these graphs. (Traditional spectral clustering
uses the normalized Laplacian matrix $\mathcal{L}$ or some similar
matrix; we use the matrix $\mathcal{L}_D$: the Laplacian restricted
to the rows and columns corresponding to non-boundary nodes.)
Moreover, unlike the traditional spectral gap, it does not trend
downwards for larger networks.  This indicates that the spectral
gap for these networks viewed as sub-graphs of an infinite graph
is finite.

There are precedents for treating networks essentially as subsets of
an an overarching infinite graph; many network generation models
\cite{ba,erdos,ws} exhibit unique convergence properties (to power-law
degree distributions or otherwise) as the size of the network grows
to infinity.  We also note that Dirichlet boundary conditions have
been shown to be successful at mitigating other boundary-related
issues in graph vertex ranking \cite{ctx}.

There is a direct connection between the spectral gap and clustering
in networks, through the Cheeger inequality.  Spectral graph theory
has led to many effective algorithms for finding cuts that result
in a small Cheeger ratio, including spectral clustering \cite{spectral1,
spectral2, survey, survey2} and local graph partitioning algorithms
\cite{acl}. These algorithms have been well-studied, both empirically
\cite{spectral1, spectral2} and theoretically \cite{spectral1,
survey2}.  Unfortunately, these algorithms can also exhibit some
undesirable behavior. It has been shown empirically \cite{lldm}
that the ``best'' partitionings of many networks, as measured by
the Cheeger ratio, result in cutting off nodes or subtrees near the
boundary of the network.  The resulting `clusters' near the boundary
actually consist of several disjoint fragments. Especially when
viewed as subsets of larger networks, this kind of clustering is
not particularly meaningful.

In this paper, we use {\it Dirichlet spectral clustering} to identify
good cuts in the networks in the Rocketfuel database. We use the
top two eigenvectors of $\mathcal{L}_D,$ the graph Laplacian with
Dirichlet boundary conditions, to cut the network into two sections.
We demonstrate that, compared to traditional spectral clustering,
there is a substantial reduction in the average number of components
resulting from the cut, without a significant increase in the Cheeger
ratio. Instead of finding  cuts near the boundaries of the networks, 
Dirichlet spectral clustering obtains cuts in the network core.

The Cheeger ratio of a cut is a well known indicator of the congestion
across the cut; small Cheeger ratios are likely to be associated
with bottlenecks.  The emphasis on identifying core bottlenecks
becomes more critical in the light of the recent observation that
many real-world graphs exhibit large-scale curvature \cite{jlb,ns}.
It has been shown \cite{ns, jlb} that such global network curvature
leads to core bottlenecks with load (or betweenness) asymptotically
much worse than flat networks, where ``load'' means the the maximum
total flow through a node assuming unit traffic between every
node-pair along shortest paths \cite{ns}. As such, it is important
to find and characterize bottlenecks at the core rather than the
fringes, where they do not matter as much. Our observations, suggest
that Dirichlet spectral clustering may be more useful in this regard.

The rest of this paper is structured as follows: in Section \ref{formal},
we give the theoretical justification for using Dirichlet eigenvalues
\cite{chung} instead of the traditional spectrum for analyzing and
clustering finite portions of infinite graphs. In Section \ref{RF},
we then compare the spectral gap using Dirichlet eigenvalues to the
traditional spectral gap on real, publicly-determined network topologies
\cite{rf} that represent smaller portions of the wider telecommunications
grid.  In Section \ref{decomp}, we demonstrate how Dirichlet spectral
clustering finds graph partitions that are more indicative of bottlenecks
in the network core rather than the fringes.

\section{Spectrum of Finite Trees: Motivation for Dirichlet Spectral Clustering}
\label{formal}
Throughout this paper, we analyze general undirected connected graphs
$G$ by using the normalized graph Laplacian $\mathcal{L}$, defined as
in \cite{book}. For two vertices $x$ and $y$, the corresponding matrix
entry is:
\[
\mathcal{L}_{xy} =
\begin{cases}
1 & \textrm{if } x = y, \\
-\frac{1}{\sqrt{d_x d_y}} & \textrm{if } x \textrm{ and } y \textrm{ are adjacent, and} \\
0 & \textrm{otherwise},
\end{cases}
\]
where $d_x$ and $d_y$ are the degrees of $x$ and $y$. We denote by
$\lambda$ the {\it spectral gap}, which is simply the smallest nonzero
eigenvalue of $\mathcal{L}$.

For any graph $G$ and finite subgraph $S \subset G$, the {\it Cheeger
ratio} $h(S)$ is a measure of the cut induced by $S$:
\[
h(S) = \frac{e(S, \bar{S})}{\min(\vol(S), \vol(\bar{S}))}.
\]
We use $e(S,\bar{S})$ to denote the number of edges crossing from $S$
to its complement, and the {\it volume} $\vol(S)$ is simply the sum of
the degrees of all nodes in $S$. The {\it Cheeger constant} $h$ is the
minimum $h(S)$ over all subsets $S$. The Cheeger constant and spectral
gap are related by the following {\it Cheeger inequality} \cite{book}:
\[
2h \geq \lambda \geq \frac{h^2}{2}.
\]
Both $\lambda$ and $h$ are often used to characterize expansion or
bottlenecks in graphs. This inequality shows that they are both good
candidates and gives the ability to estimate one based on the other.

For the infinite $d$-regular tree, the spectral gap and Cheeger constant
have both been analytically determined \cite{friedman,mckay}. Using
$\mathcal{L}$, the spectral gap is
\beeq \label{tree_spectral_gap}
\lambda = 1 - \frac{2}{d} \sqrt{d-1},
\eneq
and the Cheeger constant is $h = d-2$ \cite{hjl}.  Both of these values
are nonzero, indicating good expansion. However, the Cheeger ratio for
truncated $d$-regular trees ($TdT$) -- those with all branches of the
infinite tree cut off beyond some radius $r$ from the center -- approaches
zero as the tree gets deeper.  By cutting off any one subtree $S$ from
the root, there is only one edge connecting $S$ to $\bar{S}$, and as the
tree gets deeper, this ratio gets arbitrarily small. Using the Cheeger
inequality, it follows that the $\lambda_{TdT}\to0$ as $r\to \infty$.
Thus, the standard spectral properties of finite trees do not approach
the infinite case as they get larger; in fact, they suggest the opposite.
This is problematic when making qualitative observations about networks
and their expansion, necessitating another tool for spectral analysis
of networks.

The main reason why the traditional spectral gap does not capture
expansion well in large, finite trees is the existence of a boundary. This
is also problematic in network partitioning algorithms; often times the
``best'' partition is a {\it bag of whiskers} or combination of several
smaller cuts near the boundary \cite{lldm}. In this paper, we will use
Dirichlet eigenvalues to eliminate this problem.

Dirichlet eigenvalues are the eigenvalues of a truncated matrix,
eliminating the rows and columns that are associated with nodes on the
graph boundary. We will use a truncated normalized graph Laplacian,
$\mathcal{L}_D$, a submatrix of $\mathcal{L}$. This is different
from simply taking the Laplacian of an induced subgraph, as the edges
leading to the boundary nodes are still taken into account; it is only
the boundary nodes themselves that are ignored.  We define the {\it
Dirichlet spectral gap} to be the smallest eigenvalue of $\mathcal{L}_D.$

Using Dirichlet eigenvalues, it is also possible to obtain a {\it local
Cheeger inequality} \cite{chung} for the sub-graph $S.$ First, the {\it
local Cheeger ratio} is defined \cite{chung} for a set of nodes $T\subset
S$ as
\[
H(T) = \frac{e(T,\bar{T})}{\vol(T)};
\]
because the boundary nodes are excluded from $S$ in the definition of \cite{chung}, the
set $T$ cannot contain any boundary nodes of $S.$ The local Cheeger
ratio $H(T)$ is the appropriate quantity when $S$ is a sub-graph of a
larger graph. The {\it local Cheeger constant} $h_S$ for $S$ is then
defined as the minimum of $H(T)$ for all $T \subset S \setminus \partial(S).$
The local Cheeger inequality obtained in \cite{chung} is
\[
h_S \geq \lambda_S \geq \frac{h_S^2}{2},
\]
where $\lambda_S$ is the Dirichlet eigenvalue of the Laplacian restricted
to the rows and columns corresponding to nodes in $S$. This inequality
indicates a relationship between local expansion and bottlenecks.

The use of Dirichlet eigenvalues requires that the boundary of the
graph $S$ be defined. If $S$ is a tree, the leaf nodes are a natural
choice. When $S$ is actually a finite truncation of a larger graph,
the boundary can be defined as the set of nodes that connect directly
to other nodes outside the truncation; for the Rocketfuel data~\cite{rf},
we will use the nodes with degree $1$ which presumably connect outside
of the subnetwork.

\begin{figure}[Hht]
\begin{center}
\includegraphics[width=4in]{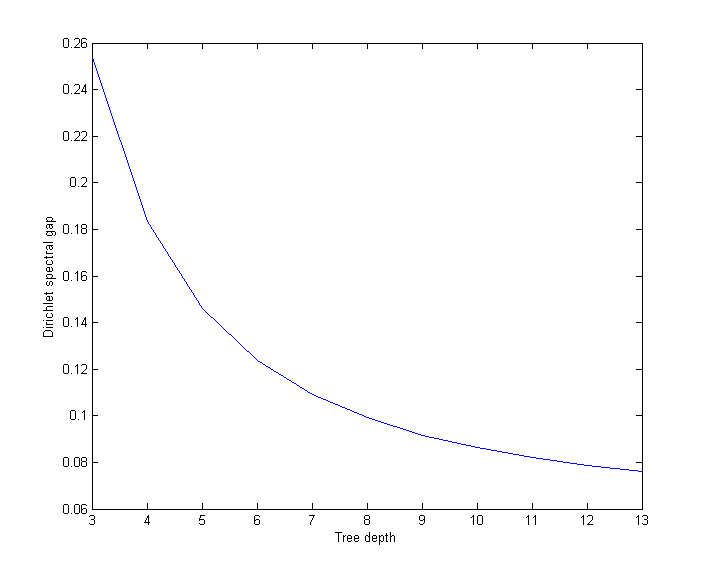} 
\caption{Dirichlet spectral gap for successively larger $3$-regular trees, showing convergence to a
nonzero value. \label{fig:convergence}}
\end{center}
\end{figure}

We first use Dirichlet eigenvalues on $d$-regular trees as prototypical
evidence for their effectiveness in capturing true spectral properties
on real-world networks. There is empirical evidence in Figure
\ref{fig:convergence}, showing that the Dirichlet spectral gap for
$3$-regular trees indeed converges to a nonzero value as tree depth
increases, contrasting with the traditional spectral gap which converges
to zero.  This is made rigorous in the following theorem:

\begin{theorem}
For finite $d$-regular trees of depth $L$, the Dirichlet spectral gap
converges to the true spectral gap (\ref{tree_spectral_gap}) of the
infinite tree as $L$ approaches infinity.
\end{theorem}

\begin{proof}
To derive the Dirichlet spectral gap for finite trees using the leaves
as the boundary, we will solve a recurrence that arises from the tree
structure and the standard eigenvalue equation
\beeq \label{ee}
\mathcal{L}_d \vec{x} = \lambda \vec{x}.
\eneq
Let $T$ be a $d$-regular tree of depth $L+1$; the $(L+1)$st level is the
boundary. We first consider eigenvectors $\vec{x}$ which have the same
value at every node at the same depth within $T$; these eigenvectors are
azimuthally symmetric.  We can represent each such eigenvector $\vec{x}$
as a sequence of values $(x_0, x_1, \ldots, x_L)$, where $x_i$ is the
uniform value at all nodes at depth $i$, similar to the analysis of
the infinite-tree spectral gap appearing in \cite{friedman}. Using this
eigenvector form for $\vec{x}$ in (\ref{ee}) leads to the recurrence:
\beeq \label{recurrence}
x_i - \frac{1}{d}x_{i-1} - \frac{d-1}{d}x_{i+1} = \lambda x_i, 2 \leq i \leq L.
\eneq
At the leaves of the tree, we have the Dirichlet boundary condition:
\beeq \label{boundary_leaf}
x_{L+1} = 0,
\eneq
and at the root of the tree we have the boundary condition
\beeq \label{type1}
x_0 - x_1 = \lambda x_0.
\eneq
We can solve (\ref{recurrence}) using the characteristic equation:
\[
\frac{d-1}{d}r^2 - (1-\lambda)r + \frac{1}{d} = 0,
\]
whose roots can be written as
\beeq \label{roots}
r_{1,2} = \frac{1}{\sqrt{d-1}} e^{\pm i \alpha}
\eneq
with
\beeq \label{dirichlet_gap}
\lambda = 1 - \frac{2}{d}\sqrt{d-1} \cos \alpha.
\eneq
Since $\lambda$ has to be real, either the real or imaginary part of
$\alpha$ must be zero.  Substituting the first boundary condition
(\ref{boundary_leaf}) yields a solution to (\ref{recurrence}) with
the form
\beeq \label{solution_form}
x_n = A (r_1^{n - L -1} - r_2^{n- L -1}).
\eneq
for some constant $A$ and $r_{1,2}$ given in (\ref{roots}). Using (\ref{type1}), 
the condition for eigenvalues is 
\begin{equation}
\frac{\tan\alpha}{\tan (L + 1)\alpha} = - \frac{d - 2}{d} \qquad 0 < \alpha < \pi.
\label{tantan}
\end{equation}
Since $\tanh x /\tanh (L + 1) x$ is positive for all real $x,$ there are 
no imaginary solutions to Eq.(\ref{tantan}). Therefore all the $L+1$ solutions
are real. From Eq.(\ref{dirichlet_gap}), the corresponding $L+1$ eigenvalues
are all outside the infinite-tree spectral gap.

We now consider eigenvectors which are zero at all nodes up to the $k$'th
level with $L > k \geq 0.$ The eigenvector is non-zero at two daughters of
some $k$'th level node and the descendants thereof.  We assume azimuthal
symmetry inside both these two sectors. The eigenvalue condition for the
parent node at the $k$'th level forces the eigenvector to be opposite in
the two sectors.  Inside each sector, (\ref{recurrence}), (\ref{roots}),
(\ref{dirichlet_gap}) (\ref{boundary_leaf}) and (\ref{solution_form})
are still valid. However, (\ref{type1}) is replaced by the condition
$x_k = 0,$ from which $\sin (L + 1 - k) \alpha = 0.$ There are $L - k$
real solutions to this equation, corresponding to eigenvalues that lie
outside the infinite-tree spectral gap, each with degeneracy $d^k (d -1).$
The total number of eigenvalues we have found so far is 
\begin{equation}
L + 1 + \sum_{k=0}^{L-1} d^k (d - 1) (L - k) = \frac{d^{L+1} - 1}{d - 1}
\end{equation}
i.e. we have found all the eigenvalues. As $L$ gets larger, the smallest 
$\alpha$ approaches $0,$ showing that the Dirichlet spectral gap converges
to the spectral gap of the infinite tree (\ref{tree_spectral_gap}) as the 
depth approaches infinity.
\end{proof}

This derivation shows that Dirichlet eigenvalues capture the expansion
properties of trees much better than the traditional spectral gap which
has been shown to approach zero for large finite trees.
This behavior on trees suggests that Dirichlet eigenvalues are a good
candidate for use in analyzing real-world networks. Such analysis appears
in Section \ref{RF}.

\section{Spectrum of Rocketfuel Networks}
\label{RF}
Our research is motivated by a series of datasets representing portions
of network topologies using Rocketfuel \cite{rf}.  Rocketfuel datasets
are publicly-available, created using {\tt traceroute} and other
networking tools to determine portions of network topology corresponding
to individual Internet service providers.  Even though like most measured
datasets, the Rocketfuel networks are not free of errors (see for example
\cite{teix}), they provide valuable connectivity information at the
IP-layer of service provider networks across the globe.  Because the
datasets were created in this manner, they represent only subsets of
the vast Internet; it becomes impossible to determine network topology
at certain points. For example, corporate intranets, home networks,
other ISP's, and network-address translation cannot be explored. The
networks used range in size from 121 to 10,152 nodes.

\begin{table}
\begin{center}
\begin{tabular}{|c|c|c|c|c|}
\hline
Dataset ID & Nodes & Edges & Traditional spectral gap & Dirichlet spectral gap \\
\hline
1221 & 2998 & 3806 & 0.00386 & 0.07616 \\
\hline
1239 & 8341 & 14025 & 0.01593 & 0.03585 \\
\hline
1755 & 605 & 1035 & 0.00896 & 0.09585 \\
\hline
2914 & 7102 & 12291 & 0.00118 & 0.04621 \\
\hline
3257 & 855 & 1173 & 0.01045 & 0.04738 \\
\hline
3356 & 3447 & 9390 & 0.00449 & 0.05083 \\
\hline
3967 & 895 & 2070 & 0.00799 & 0.03365 \\
\hline
4755 & 121 & 228 & 0.03570 & 0.06300 \\
\hline
6461 & 2720 & 3824 & 0.00639 & 0.11036 \\
\hline
7018 & 10152 & 14319 & 0.00029 & 0.09531 \\
\hline
Grid & 10000 & 19800 & 0.00025 & 0.00050 \\
\hline
\end{tabular}
\caption{Structural and spectral properties of Rocketfuel datasets. \label{table:rf}}
\end{center}
\end{table}

Because of the method of data collection, the Rocketfuel datasets
contain many degree-$1$ nodes that appear at the edge of the topology.
In actuality, the network extends beyond this point, but the datasets
are limited to one ISP at a time.  As such, it makes sense to view
these degree-$1$ nodes as the boundary of a finite subset of a much
larger network. Using this boundary definition, we compute the
Dirichlet spectral of these graphs and compare with their standard
counterparts, as shown in Table \ref{table:rf} and Figure~\ref{fig:gaps}.  It is apparent
that the Dirichlet spectral gaps are much larger than the traditional
spectral gaps for all the networks, implying a much higher degree
of expansion than one would traditionally obtain.
The spectral gaps for a two-dimensional square
Euclidean grid are also shown; the grid is known to be a poor expander, and accordingly
even the Dirichlet spectral gap is very small.
\begin{figure}[Hht]
\begin{center}
\includegraphics[width=4in]{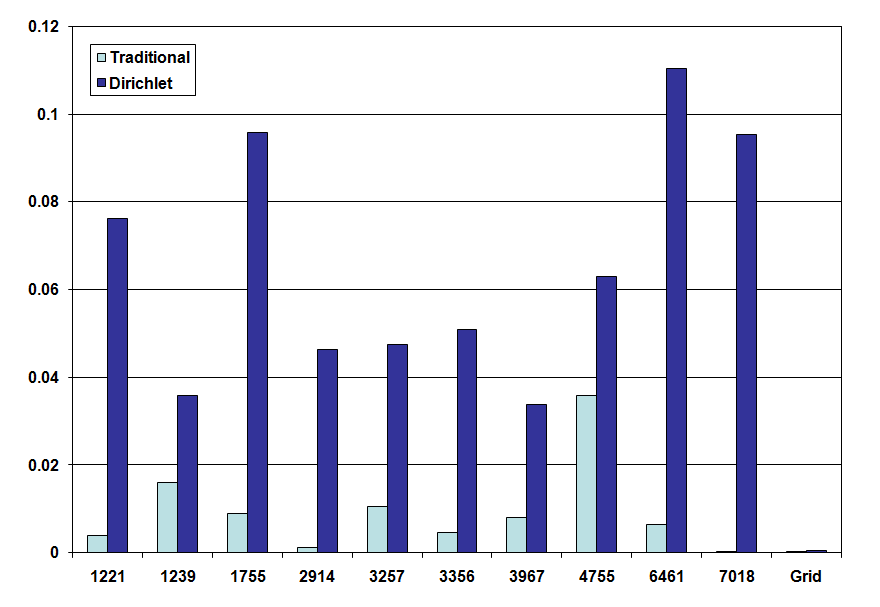} 
\caption{Comparison of traditional and Dirichlet spectral gaps in Rocketfuel data as well as the
$2$-dimensional Euclidean grid.
\label{fig:gaps}}
\end{center}
\end{figure}

Figure~\ref{fig:lngaps} shows the same data, plotted as a function of the 
number of nodes $N$ in each network. We see that the traditional spectral
gap keeps decreasing as $N$ is increased, whereas the Dirichlet spectral
gap does not.

\begin{figure}[Hht]
\begin{center}
\includegraphics[width=4in]{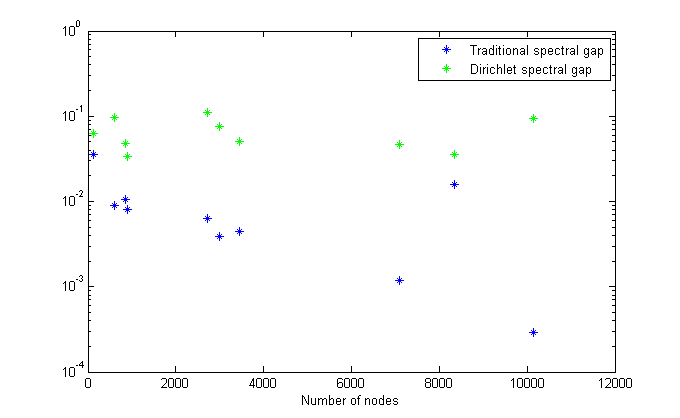} 
\caption{Comparison of traditional and Dirichlet spectral gaps across Rocketfuel networks.
\label{fig:lngaps}}
\end{center}
\end{figure}

Since Figure~\ref{fig:lngaps} compares different networks, possibly with different 
properties, we confirm the result by computing the spectral gap for subgraphs
of different sizes drawn from a single network. All the nodes that are within
a distance $r$ of the center of mass of a network are included in a subgraph,
with $r$ varying between 1 and the maximum possible value for the network. In 
Fig.~\ref{fig:growing} shows the results for the largest of the Rocketfuel
networks, dataset 7018 containing over 10,000 nodes. 
For a subgraph of radius $r$, the boundary
is defined as all the nodes which i) have edges connecting them to nodes in 
the graph that are outside the subgraph or ii) connect to the outside world, 
i.e. that have degree 1 in the full dataset. As in Fig.~\ref{fig:lngaps}, in 
Fig.~\ref{fig:growing}, the traditional spectral gap keeps decreasing as $r$
is increased, but the Dirichlet spectral gap does not.
\begin{figure}[Hht]
\begin{center}
\includegraphics[width=4in]{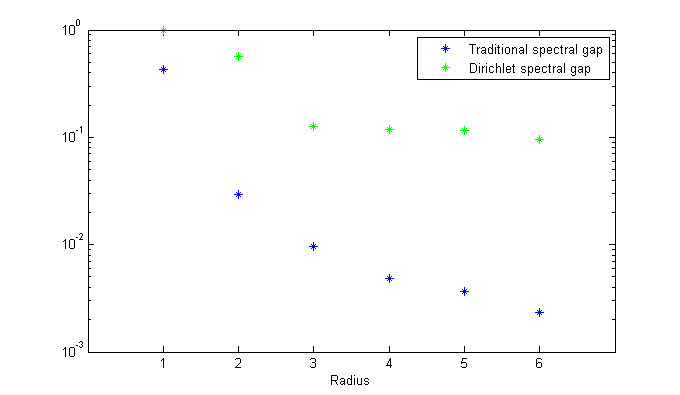} 
\caption{Comparison of traditional and Dirichlet spectral gaps in successively larger subgraphs,
grown from the center of mass of dataset 7018.
\label{fig:growing}}
\end{center}
\end{figure}

\section{Spectral Decomposition}
\label{decomp}
One important application of the eigendecomposition of a graph is spectral
clustering or partitioning \cite{spectral1,spectral2}.  The problem is
to group the nodes into partitions, clusters, or communities that are
inherently well-connected within themselves, with sparser connections
between clusters. This is closely related to finding bottlenecks; if a
graph has a bottleneck, then a good partition is often found by dividing
the graph at the bottleneck. See \cite{survey} for a general survey of
graph clustering.

It is often desirable for a network partition to be balanced, and
finding bottlenecks near the core or center of mass of a network is
often more useful than simply clipping small subsets of nodes near the
boundary. But according to \cite{lldm}, using the Cheeger ratio as a
metric on real-world data, the ``best'' cuts larger than a certain
critical size are actually ``bags of whiskers'' or combinations of
numerous smaller cuts. Because many graph clustering algorithms,
including spectral clustering, try to optimize for this metric, the
resulting partitions often slice numerous smaller cuts off the graph,
which is not always useful. For our Rocketfuel data, we know that the
boundary of the network is imposed by the method of data collection. Thus,
by eliminating the boundary from graph clustering, we can more easily
find partitions that are more evenly balanced, and bottlenecks that are
closer to the core of the network.

To do this, we use standard spectral clustering techniques from
\cite{spectral1}, but instead of using the normalized graph Laplacian
$\mathcal{L}$, we use the truncated Dirichlet version $\mathcal{L}_D$. The
eigenvectors used for clustering will therefore not include components
for the degree-$1$ boundary nodes, but we can assign them to the same
side of the partition as their non-boundary neighbor nodes. Specifically,
we compute the first two eigenvectors of $\mathcal{L}_D$ and cluster the
nodes based on their components in these eigenvectors using $k$-means. For
each node, we compute the distance to both centers and sort the nodes
based on the difference. For a partition of size $k$, we take the top
$k$ nodes.

We follow the experiments of Leskovec et al. in \cite{lldm} by using
both traditional spectral clustering and Dirichlet spectral clustering
to find cuts of different sizes. Specifically, we find Dirichlet cuts
of all possible sizes, and then we find cuts using traditional spectral
clustering for those same sizes after adding boundary nodes back in. Thus,
for each network of $N$ nodes, we calculate $N-B$ cuts, where $B$ is
the number of boundary nodes.

For each cut, we measure the Cheeger ratio $h$ and the number of
components $c$. Ideally, a logical cut would split the network into
exactly $c=2$ components, but as Leskovec et al. demonstrated, as cut
size increases, spectral clustering and other algorithms that optimize
for $h$ yield cuts with many components.  This is precisely the problem
we are trying to avoid using Dirichlet clustering, and our results
show that Dirichlet clustering is effective in finding cuts with fewer
components. Furthermore, even though our algorithm is not specifically
optimizing for $h$, it does not find cuts that have significantly worse
values for $h$ while finding cuts with far fewer components.

\begin{table}[ht]
\begin{center}
\begin{tabular}{|c|c|c|c|c|r|r|r|r|}
\hline
& \multicolumn{4}{c|}{Number of cuts in each category:} & & & & \\
\cline{2-5}
& $c_D \leq c_T$ & $c_D \leq c_T$ & $c_D > c_T$ & $c_D > c_T$ & Avg & Avg & Avg & Avg \\
Dataset & $h_D \leq h_T$ & $h_D > h_T$ & $h_D \leq h_T$ & $h_D > h_T$ & $c_D - c_T$ & $h_D - h_T$ & $c_T$ & $h_T$ \\
\hline
1221 & 49 & 197 & 0 & 6 & -28.9 & 0.0506 & 36.8 & 0.0829 \\
1239 & 538 & 362 & 0 & 30 & -75.1 & 0.0127 & 83.4 & 0.1326 \\
1755 & 32 & 91 & 0 & 14 & -4.5 & 0.0545 & 7.9 & 0.1210 \\
2914 & 224 & 819 & 0 & 323 & -107.3 & 0.0565 & 125.8 & 0.1639 \\
3257 & 49 & 67 & 0 & 35 & -12.3 & 0.0370 & 20.0 & 0.1386 \\
3356 & 182 & 315 & 3 & 41 & -34.6 & 0.0388 & 45.6 & 0.1895 \\
3967 & 24 & 137 & 3 & 129 & -3.2 & 0.1423 & 9.2 & 0.1215 \\
4755 & 15 & 6 & 0 & 6 & -12.3 & -0.0970 & 15.4 & 0.3460 \\
6461 & 111 & 199 & 0 & 73 & -13.4 & 0.0148 & 19.7 & 0.0999 \\
7018 & 157 & 465 & 12 & 273 & -54.3 & 0.0403 & 81.4 & 0.0735 \\
\hline
\end{tabular}
\caption{Aggregate data comparing Dirichlet spectral clustering with traditional spectral clustering for
         several Rocketfuel datasets. For each dataset, we compute Dirichlet cuts of all possible sizes, and
         compare them with traditional spectral cuts with the same sizes. Smaller values of $c$ and $h$ are
         better. We classify the cuts into four categories, counting the number in each, and we also give the
         average difference in $h$ and $c$ between Dirichlet and traditional spectral clustering. The data shows
         that Dirichlet clustering finds cuts with many fewer components without significant adverse effects on
         the Cheeger ratio.}
\label{table}
\end{center}
\end{table}

We outline some aggregate data in Table \ref{table}. For several
datasets, we count the number of cuts in four different categories,
comparing the Dirichlet Cheeger ratio and number of components ($h_D$
and $c_D$) with traditional spectral clustering ($h_T$ and $c_T$). It is
evident that Dirichlet clustering finds cuts with fewer components than
traditional spectral clustering ($c_D \leq c_T$) for most cut sizes,
indicating that while spectral clustering optimizes for Cheeger ratio,
it often ``cheats'' by collecting whiskers as one cut. In addition,
despite the use of Cheeger ratio optimization, Dirichlet clustering
sometimes finds cuts with better Cheeger ratio as well. In the last
two columns for each dataset, we give the difference in $h$ and $c$
averaged out over all cut sizes. It turns out that the Cheeger ratios,
on average, are not drastically different between the two methods,
and Dirichlet clustering gives cuts with far fewer components.

Along with our aggregate data, we illustrate each individual cut for several of our
Rocketfuel datasets in Fig.\ \ref{fig:comparison}. (A few of the 
datasets were too large for accurate numerical computation.) For each cut size,
we plot a point corresponding to the difference in Cheeger ratio $h$
and the number of components $c$ between Dirichlet and traditional
spectral clustering.  It should be clear that for the majority of cut
sizes, Dirichlet clustering finds cuts with far fewer components, but
there is generally little change in Cheeger ratio. This can be seen
in the large variation on the $c$-axis with much smaller discrepancies
on the $h$-axis.  In other words, Dirichlet clustering avoids finding
``bags of whiskers'' while still maintaining good separation in terms
of $h$, despite not explicitly optimizing for $h$.

\begin{figure*}[htf]
\centering
\subfloat[Dataset 1221]{\label{fig:1221}\includegraphics[width=0.45\textwidth]{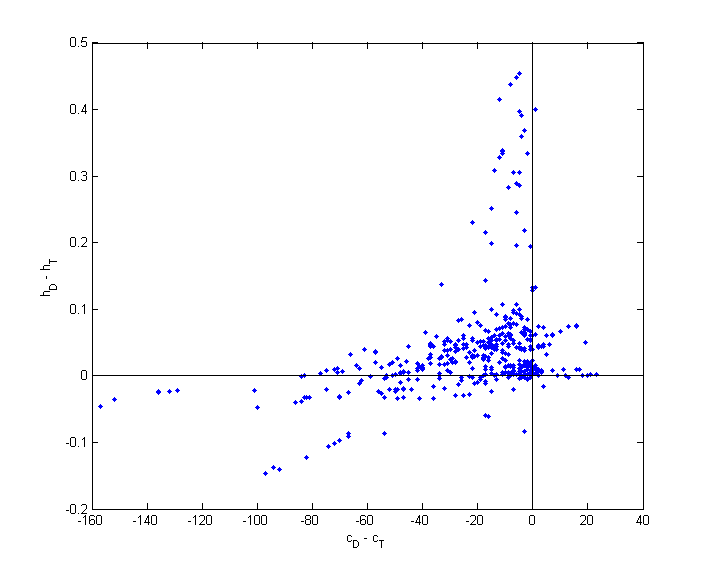}}
\subfloat[Dataset 1755]{\label{fig:1755}\includegraphics[width=0.45\textwidth]{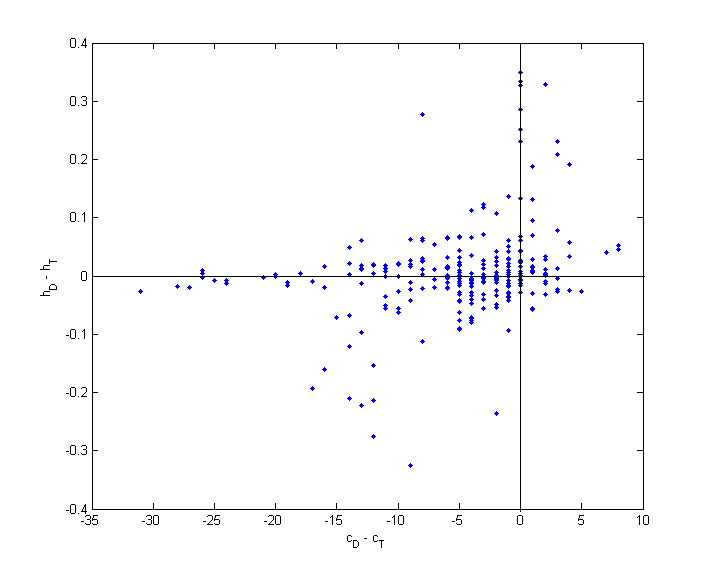}} \\
\subfloat[Dataset 3257]{\label{fig:3257}\includegraphics[width=0.45\textwidth]{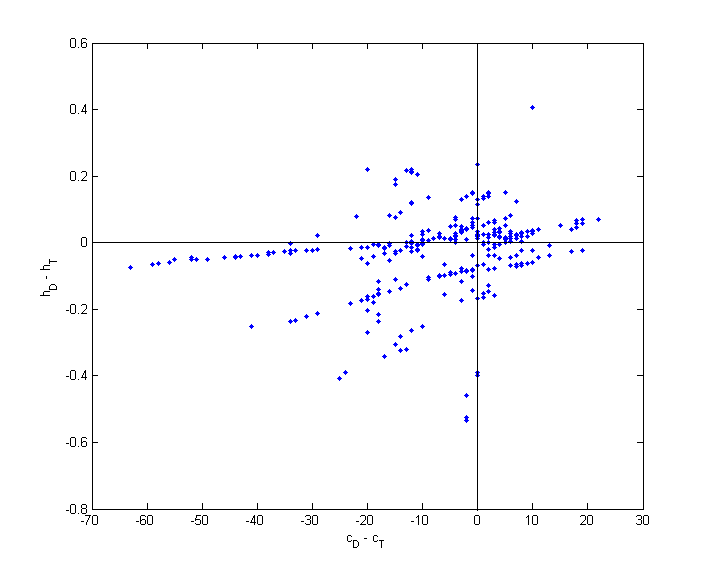}}
\subfloat[Dataset 3356]{\label{fig:3356}\includegraphics[width=0.45\textwidth]{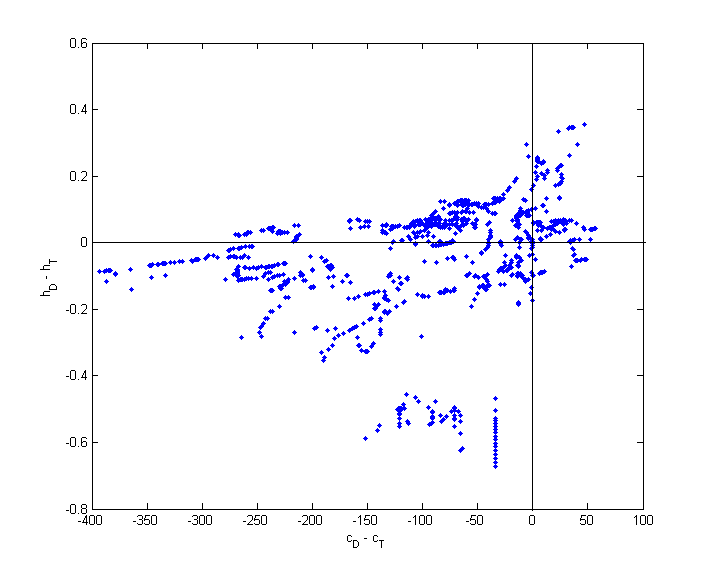}} \\
\subfloat[Dataset 3967]{\label{fig:3967}\includegraphics[width=0.45\textwidth]{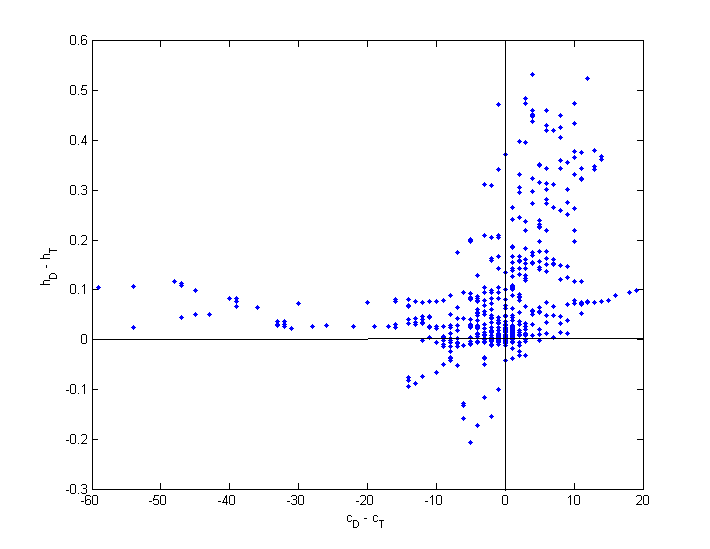}}
\subfloat[Dataset 6461]{\label{fig:6461}\includegraphics[width=0.45\textwidth]{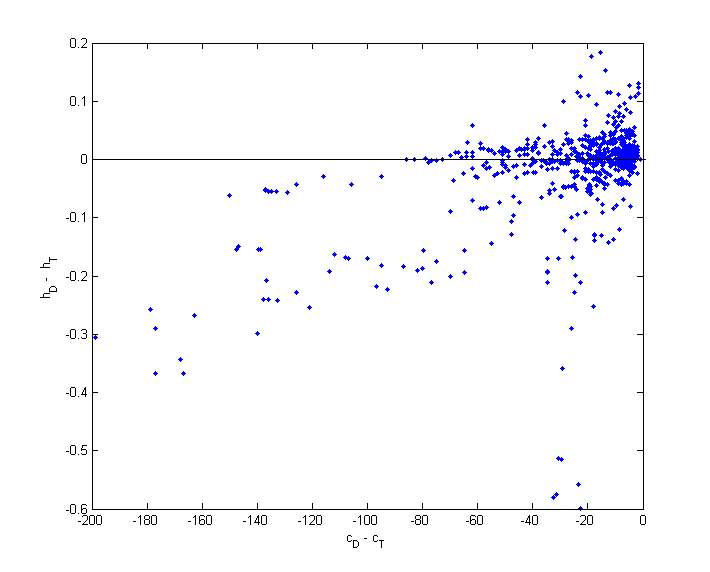}}
\caption{Comparison of Cheeger ratio $h$ and number of components $c$ for cuts for various datasets using
         Dirichlet ($D$) and traditional ($T$) spectral clustering. Each point represents one possible
         cut size; in general, Dirichlet clustering yields many fewer components without sacrificing much
         in Cheeger ratio.}
\label{fig:comparison}
\end{figure*}

It is clear that using Dirichlet eigenvalues improves the partition
by ignoring the boundary, alleviating the tendency to find ``bags of
whiskers'' without drastically changing the Cheeger ratio.  Although
traditional spectral clustering does not always fail, there is clear
evidence that Dirichlet spectral properties are an important tool in
the analysis of real-world networks.

\section{Discussion}
\label{conc}
Our results show evidence that eigenvalues of the graph Laplacian
can provide rich information about real-world networks when Dirichlet
boundary conditions are applied.  We find that the Dirichlet spectral
gap computed for several IP-layer networks is much larger than the
traditional spectral gap, and is likely to go to a finite limit as
the size of the network is increased.  Rigorous analysis for infinite
$d$-regular trees suggests that this may be the same as the spectral
gap of an infinite communications network. Spectral clustering using 
Dirichlet eigenvalues yields much better clustering than traditional
methods. 

The spectral decomposition using Dirichlet eigenvalues also suggests
a connection to large-scale negative curvature \cite{jlb,jlbb,ns}
in the Rocketfuel data. Traditional negatively curved graphs such
as trees and hyperbolic grids generally exhibit poor connectivity
and core congestion. Standard clustering often yields combinations
of smaller cuts near the periphery of the graph, but using Dirichlet
clustering, we can see that there tend to be bad larger-scale cuts
as well in the Rocketfuel datasets, in the graph interior.  The
presence of these larger-scale cuts is a hallmark of negative
curvature or hyperbolicity \cite{gromov}, suggesting that Dirichlet
spectral clustering may yield different behavior for hyperbolic and
flat networks. The hyperbolic grids themselves are also suitable
for further analysis, building from our study of regular trees.
Many properties such as the spectral gap remain open questions.

With some evidence of a connection between global negative curvature,
the spectral gap, and expansion, it would be interesting to empirically
compare the hyperbolicity $\delta$, the Cheeger constant $h$, and the
traditional and Dirichlet spectral gaps of Rocketfuel and other real-world
networks as well as well-known network models. From this, it could be
possible to classify various networks based on these properties.


\begin{thebibliography}{99}
\bibitem{acl}
R. Andersen, F. Chung and K. Lang, Local graph partitioning using PageRank vectors.
{\it Proceedings of the 47th Annual IEEE Symposium on Foundations of Computer Science (FOCS 2006)}, 475--486.
\bibitem{ba}
A.-L. Barab\'{a}si and R. Albert, Emergence of scaling in random networks.
{\it Science} 286 (1999), 509--512.
\bibitem{chung}
F. Chung, Random walks and local cuts in graphs. {\it Linear Algebra and its Applications} 423 (2007), 22--32.
\bibitem{book}
F. Chung, Spectral Graph Theory. Providence, RI: American Mathematical Society, 1997.
\bibitem{ctx}
F. Chung, and A. Tsiatas and W. Xu, Dirichlet PageRank and trust-based ranking algorithms. {\it Proceedings
of the Workshop on Algorithms and Models for the Web Graph (WAW 2011)},
{\it Lecture Notes in Computer Science} 6732, 103--114.
\bibitem{erdos}
P. Erd\"{o}s and A. R\'{e}nyi, On the evolution of random graphs.
{\it Publications of the Mathematical Institute of the Hungarian Academy of Sciences} 5 (1960), 17--61.
\bibitem{friedman}
J. Friedman, The spectra of infinite hypertrees. {\it SIAM Journal on Computing} 20 (1991), 951--961.
\bibitem{gromov}
M. Gromov, Essays in Group Theory. {\it Mathematical Sciences Research Institute Publications} 8 (1987), 75.
\bibitem{hjl}
O. H\"{a}ggstr\"{o}m, J. Jonasson and R. Lyons, Explicit isoperimetric constants and phase transitions in
the random-cluster model. {\it Annals of Probability} 30 (2002), 443--473.
\bibitem{jlb}
E. Jonckheere, P. Lohsoonthorn and F. Banahon, Scaled Gromov hyperbolic graphs. {\it Journal of Graph Theory} 57 (2008), 157--180.
\bibitem{jlbb}
E. Jonckheere, M. Lou, F. Bonahon and Y. Baryshnikov, Euclidean versus hyperbolic congestion in idealized
versus experimental networks.  ArXiv e-print 0911.2538, 2009.
\bibitem{lldm}
J. Leskovec, K. Lang, A. Dasgupta and M. Mahoney, Statistical properties of community structure in large
social and information networks. {\it Proceedings of the 17th International Conference on the World Wide Web (WWW 2008)}, 695--704.
\bibitem{mckay}
B. McKay,
The Expected Eigenvalue Distribution of a Large Regular Graph,
{\it Linear Algebra and Applications}, 40 (1981), 203--216.
\bibitem{ns}
O. Narayan and I. Saniee, The large scale curvature of networks. ArXiv e-print 0907.1478, 2009.
\bibitem{spectral1}
A. Ng, M. Jordan and Y. Weiss, On spectral clustering: analysis and an algorithm.  {\it Advances in Neural Information Processing Systems} 14 (2002), 849--856.
\bibitem{survey}
S. E. Schaeffer, Graph clustering. {\it Computer Science Review} 1 (2007), 27--64.
\bibitem{spectral2}
J. Shi and J. Malik, Normalized cuts and image segmentation. {\it IEEE Transactions on Pattern Analysis and Machine Intelligence} 22 (2000), 888--905.
\bibitem{rf}
N. Spring, R. Mahajan and D. Wetherall, Measuring ISP topologies with Rocketfuel. {\it Proceedings of the
2002 SIGCOMM Conference}, 133--145.
\bibitem{teix}
R. Teixeira, K. Marzullo, S. Savage and G. M. Voelker,
In search of path diversity in ISP networks. {\it Proceedings of the
2003 SIGCOMM Conference}, 2003.
\bibitem{survey2}
U. von Luxburg, A tutorial on spectral clustering. {\it Statistics and Computing} 17 (2007), 395--416.
\bibitem{ws}
D. Watts and S. Strogatz, Collective dynamics of `small-world' networks.
{\it Nature} 393 (1998), 440--442.
\end{thebibliography}
\end{document}